\documentclass[11pt]{amsart}

\usepackage{amsmath}
\usepackage{amsthm}
\usepackage{amsfonts}
\usepackage{bbm}
\usepackage{kpfonts}
\usepackage{amssymb,enumitem}
\usepackage{color}
\usepackage{xcolor}
\usepackage{hyperref}

\newcommand{\LCM}{{\rm LCM}\,}
\newcommand{\GCD}{{\rm GCD}\,}

\renewcommand{\mod}{{\rm mod}\,}

\newcommand{\E}{\mathbb{E}}
\newcommand{\R}{\mathbb{R}}
\newcommand{\todistrn}{\overset{w}{\underset{n\to\infty}\longrightarrow}}

\newcommand{\N}{\mathbb{N}}
\renewcommand{\R}{\mathbb{R}}
\newcommand{\mmp}{\mathbb{P}}

\DeclareMathOperator{\1}{\mathbbm{1}}

\newtheorem{thm}{Theorem}[section]
\newtheorem{lemma}[thm]{Lemma}

\newtheorem{cor}[thm]{Corollary}
\newtheorem{example}[thm]{Example}
\newtheorem{assertion}[thm]{Proposition}
\theoremstyle{definition}
\newtheorem{define}[thm]{Definition}
\theoremstyle{remark}
\newtheorem{rem}[thm]{Remark}

\begin{document}

\title[{Multivariate multiplicative functions of random vectors in large integer domains}]{Multivariate multiplicative functions of uniform random vectors in large integer domains}

\author{Zakhar Kabluchko}
\address{Zakhar Kabluchko: Institut f\"ur Mathematische Stochastik, Westf\"alische Wilhelms-Universit\"at M\"unster, M\"unster, Germany}
\email{zakhar.kabluchko@uni-muenster.de}

\author{Oleksandr Marynych}
\address{Oleksandr Marynych: Faculty of Computer Science and Cybernetics, Taras Shev\-chen\-ko National University of Kyiv, Kyiv, Ukraine}
\email{marynych@unicyb.kiev.ua}

\author{Kilian Raschel}
\address{Kilian Raschel: Universit\'{e} d’Angers, CNRS, Laboratoire Angevin de Recherche en Math\'{e}matiques, Angers, France}
\email{raschel@math.cnrs.fr}

\begin{abstract}
For a wide class of sequences of integer domains $\mathcal{D}_n\subset\N^d$, $n\in\N$, we prove distributional limit theorems for $F(X_1^{(n)},\ldots,X_d^{(n)})$, where $F$ is a multivariate multiplicative function and $(X_1^{(n)},\ldots,X_d^{(n)})$ is a random vector with uniform distribution on $\mathcal{D}_n$. As a corollary, we obtain limit theorems for the greatest common divisor and least common multiple of the random set $\{X_1^{(n)},\ldots,X_d^{(n)}\}$. This generalizes previously known limit results for $\mathcal{D}_n$ being either a discrete cube or a discrete hyperbolic region.
\end{abstract}



\maketitle

\section{Introduction}
Let $F:\mathbb{N}^d\to \mathbb{C}$ be an arithmetic function of $d\geq1$ integer arguments, with $\mathbb N=\{1,2,3,\ldots\}$. A standard problem in analytic number theory is the estimation of the multivariate sum
$$
\sum_{x_1=1}^{n_1}\cdots\sum_{x_d=1}^{n_d}F(x_1,\ldots,x_d)
$$
for large values of $(n_1,\ldots,n_d)\in\mathbb{N}^d$. A particular instance of this problem consists in establishing existence of the so-called mean value of $F$, which is defined via
\begin{equation}\label{eq:mean_value_gen_def}
M(f):=\lim_{n_1,\ldots,n_d\to\infty}\frac{1}{n_1 \cdots n_d}\sum_{x_1=1}^{n_1}\cdots\sum_{x_d=1}^{n_d}F(x_1,\ldots,x_d).
\end{equation}
In the probabilistic language, \eqref{eq:mean_value_gen_def} may be recast as follows. Let $(U_1^{(n_1)},\ldots,U_d^{(n_d)})$ be a random vector defined on some probability space $(\Omega,\mathcal{F},\mathbb{P})$ and which has the uniform distribution on the finite rectangular set
\begin{equation}
\label{eq:rectangular_domain_def}
   \mathcal{R}_{n_1,\ldots,n_d}:=\left(\varprod_{i=1}^{d}[1,n_i]\right)\bigcap \N^d.
\end{equation}
Then, with $\E$ denoting the expectation with respect to $\mathbb{P}$,
\begin{equation}\label{eq:mean_value_gen_def_probab}
M(F)=\lim_{n_1,\ldots,n_d\to\infty}\E F(U_1^{(n_1)},\ldots,U_d^{(n_d)}).
\end{equation}
A general result on existence of $M(F)$ is due to Ushiroya~\cite{Ushiroya:2012}.

A multivariate arithmetic function $F:\mathbb{N}^d\to\mathbb{C}$ is called multiplicative, see \cite{Toth:2014,Ushiroya:2012,Vaidyanathaswamy:1931}, if
$$
F(1,\ldots,1)=1\quad\text{and}\quad F(m_1n_1,\ldots,m_d n_d)=F(m_1,\ldots,m_d)F(n_1,\ldots,n_d),
$$
for all $(m_1,\ldots,m_d)\in\N^d$ and $(n_1,\ldots,n_d)\in\N^d$ such that
$$
\GCD(m_1 \cdots m_d,n_1 \cdots n_d)=1.
$$
A specialization of Ushiroya's results from \cite{Ushiroya:2012} to a multiplicative function $F$ implies that under a mild summability assumption on $F$, the mean value $M(F)$ exists and is equal to
\begin{equation}
\label{eq:mean_value_multiplicative}
   M(F):=\prod_{p\in\mathcal{P}}\left(1-\frac{1}{p}\right)^d \sum_{i_1=0}^{\infty}\cdots\sum_{i_d=0}^{\infty}\frac{F(p^{i_1},\ldots,p^{i_d})}{p^{i_1+\cdots+i_d}},
\end{equation}
where $\mathcal{P}$ stands for the set of prime numbers.

In the last years, there has been a lot of activity around various generalizations and extensions of the aforementioned results. In a probabilistic direction, one may ask about the asymptotic behavior of {\it distributions} of the random variable $F(U_1^{(n_1)},\ldots,U_d^{(n_d)})$, as $n_1,\ldots,n_d\to\infty$ in \eqref{eq:rectangular_domain_def}. This question has been addressed in~\cite{BosMarRas:2019} for a particular choice of $F$, namely, for
$F(x_1,\ldots,x_d)=G(\LCM(x_1,\ldots,x_d))$, with $G$ being a univariate multiplicative arithmetic function. The univariate case $d=1$ is the classical Erd\H{o}s-Wintner theorem, see \cite{EW:1939}, which provides necessary and sufficient conditions for the distributional convergence of $F(U_1^{(n)})$ as $n\to\infty$. In another, more analytic direction, the rectangular domains $\mathcal{R}_{n_1,\ldots,n_d}$ in \eqref{eq:rectangular_domain_def} are replaced by more sophisticated domains of summation $\mathcal{D}_n\subset\N^d$, which grow to $\N^d$ as $n\to\infty$. In particular, in the recent work~\cite{HeyToth:2022-2}, the case of \textit{spherical} summation over the regions
$$
\mathcal{S}_n:=\{(x_1,\ldots,x_d)\in\N^d:x_1^2+\cdots+x_d^2\leq n\},
$$
has been analyzed, whereas the papers~\cite{IksMarRas:2022,HeyToth:2021,HeyToth:2022-1} were devoted to the study of summation over \textit{hyperbolic} regions
$$
\mathcal{H}_n:=\{(x_1,\ldots,x_d)\in\N^d:x_1 \cdots x_d\leq n\}
$$
and their generalizations. A surprising phenomenon revealed in the cited works is that the mean value $M(F)$ given by~\eqref{eq:mean_value_multiplicative} is universal for rectangular, spherical and hyperbolic domains. More specifically, let $\mathcal{D}_n$ be either $\mathcal{R}_{n,\ldots,n}$, $\mathcal{S}_n$ or $\mathcal{H}_n$. For every $n\in\N$, let $(X_1^{(n)},\ldots,X_d^{(n)})$ be a random vector defined on $(\Omega,\mathcal{F},\mathbb{P})$ and having the uniform distribution on $\mathcal{D}_n$, that is,
$$
\mmp\{(X_1^{(n)},\ldots,X_d^{(n)})=(i_1,\ldots,i_d)\}=\frac{1}{\#\mathcal{D}_n},\quad (i_1,\ldots,i_d)\in\mathcal{D}_n,
$$
where $\#\mathcal{D}_n$ denotes the cardinality of $\mathcal{D}_n$. Then, under the same summability assumption on $F$ as in Ushiroya's  result, we have
\begin{multline}\label{eq:universality}
\lim_{n\to\infty}\E F(X_1^{(n)},\ldots,X_d^{(n)})
=\lim_{n\to\infty}\frac{1}{\#\mathcal{D}_n}\sum_{(x_1,\ldots,x_d)\in\mathcal{D}_n}F(x_1,\ldots,x_d)
\\
=M(F)=\prod_{p\in\mathcal{P}}\left(1-\frac{1}{p}\right)^d \sum_{i_1=0}^{\infty}\cdots\sum_{i_d=0}^{\infty}\frac{F(p^{i_1},\ldots,p^{i_d})}{p^{i_1+\cdots+i_d}}.
\end{multline}

The purpose of the present paper is two-fold. First, we shall provide a probabilistic explanation which lies in the core of~\eqref{eq:universality}, by providing sufficient conditions on $F$ for the distributional convergence of $F(X_1^{(n)},\ldots,X_d^{(n)})$ as $n\to\infty$. Second, we shall do this not only for the three types of regions mentioned before, but for a quite general class of integer domains $\mathcal{D}_n$ satisfying mild assumptions.

The paper is organized as follows. In Section~\ref{sec:results}, we formulate our standing assumptions on $\mathcal{D}_n$ and present our main results, which are distributional limit theorems for $F(X_1^{(n)},\ldots,X_d^{(n)})$. The proofs are collected in Section~\ref{sec:proofs}. In Section~\ref{sec:examples_domains}, we provide various examples of domains $\mathcal{D}_n$ satisfying our standing assumptions. In particular, the aforementioned domains $\mathcal{R}_{n_1,\ldots,n_d}$, $\mathcal{S}_n$ and $\mathcal{H}_n$ are covered. In Section~\ref{sec:set-theoretic-operations} we discuss how to construct new domains satisfying our conditions, using standard set-theoretic operations. Some auxiliary results are collected in Appendix~\ref{sec:appendix}. 

Throughout the paper we use the following standard notation: $\overset{w}{\longrightarrow}$ denotes the convergence in distribution (weak convergence of probability measures); ${\rm Int}(A)$, ${\rm cl}(A)$ and $\partial A$ are the topological interior, closure and boundary of a set $A\subset \R^d$, respectively; $a(n)\sim b(n)$, $n\to\infty$, means that $\lim_{n\to\infty}(a(n)/b(n))=1$.

\section{Main results}\label{sec:results}
\subsection{Preliminaries}
Throughout the paper, we assume that $F$ is a multivariate multiplicative arithmetic function of $d\geq 2$ variables. Every multivariate multiplicative function is completely determined by its values on the powers of primes. More precisely, let $\lambda_p(n)$ denote the power of prime $p\in\mathcal{P}$ in the prime decomposition of $n\in\N$. Then
$$
x_i=\prod_{p\in\mathcal{P}}p^{\lambda_p(x_i)},\quad i=1,\ldots,d,
$$
implies
$$
F(x_1,\ldots,x_d)=\prod_{p\in\mathcal{P}}F(p^{\lambda_p(x_1)},\ldots,p^{\lambda_p(x_d)}).
$$

The crucial observation for everything to follow is the  representation for $M(F)$ in~\eqref{eq:mean_value_multiplicative} via independent geometric random variables. Let $(\mathcal{G}_1(p),\ldots,\mathcal{G}_d(p))_{p\in\mathcal{P}}$ be an array of mutually independent random variables with geometric distributions
$$
\mmp\{\mathcal{G}_k(p)\geq j\}=\frac{1}{p^j},\quad j\in\N_0,\quad p\in\mathcal{P},\quad k=1,\ldots,d,
$$
where $\N_0:=\N\cup\{0\}$. Then
$$
M(F)=\E\left(\prod_{p\in\mathcal{P}}F(p^{\mathcal{G}_1(p)},\ldots,p^{\mathcal{G}_d(p)})\right).
$$
The main result of our paper gives sufficient conditions on $F$ which ensure the convergence in distribution
\begin{equation}
\label{eq:main_convergence_F}
F(X_1^{(n)},\ldots,X_d^{(n)})=\prod_{p\in\mathcal{P}}
F(p^{\lambda_p(X_1^{(n)})},\ldots,p^{\lambda_p(X_d^{(n)})})\todistrn
\prod_{p\in\mathcal{P}}F(p^{\mathcal{G}_1(p)},\ldots,p^{\mathcal{G}_d(p)})=:F_{\infty},
\end{equation}
for a general class of integer domains $\mathcal{D}_n$, which we are now going to introduce. 

Let $(\mathcal{D}_n)_{n\in\N}$ be a sequence of finite, non-empty subsets of $\N^d$. Assume that for every fixed $c\in\mathbb{Z}^d$, where $\mathbb{Z}=\{0,\pm 1,\pm 2,\ldots\}$, the following condition is fulfilled:
\begin{equation}\label{eq:van_Hove}
\lim_{n\to\infty}\frac{\#((\mathcal{D}_n+c)\cap \mathcal{D}_n)}{\#\mathcal{D}_n}=1.
\end{equation}
Note that~\eqref{eq:van_Hove} is equivalent to saying that for all $c\in\mathbb{Z}^d$,
$$
\lim_{n\to\infty} \frac{\delta_n(c)}{\#\mathcal{D}_n}=0,
$$
where, denoting $\Delta$ the symmetric difference of two sets,
\begin{equation}
\label{eq:delta}
   \delta_n(c):=\# (\mathcal{D}_n\Delta (\mathcal{D}_n+c)).
\end{equation}
Condition~\eqref{eq:van_Hove} is known in the literature as the regular growth condition; see Chapter~3 in~\cite{Bulinski+Shashkin}. Several equivalent versions  of~\eqref{eq:van_Hove} can be found in Appendix~\ref{sec:appendix} below.

\subsection{Convergence of prime powers to geometric laws}
Our first main result states that, solely under assumption~\eqref{eq:van_Hove}, the array of random vectors $(\lambda_p(X_1^{(n)}),\ldots,\lambda_p(X_d^{(n)}))_{p\in\mathcal{P}}$ converges in distribution to an array of independent geometric variables, thereby providing the first evidence supporting~\eqref{eq:main_convergence_F}.
\begin{thm}\label{prop:geometric}
Assume that~\eqref{eq:van_Hove} holds. Then
$$
\left(\lambda_p(X_1^{(n)}),\ldots,\lambda_p(X_d^{(n)})\right)_{p\in\mathcal{P}}~\todistrn~\left( \mathcal{G}_1(p),\ldots,\mathcal{G}_d(p)\right)_{p\in\mathcal{P}},
$$
in the space $(\R^d)^{\infty}$ endowed with the product topology.
\end{thm}

\begin{rem}
In the rectangular case $\mathcal{D}_n=\mathcal{R}_{n_1,\ldots,n_d}$, Theorem~\ref{prop:geometric} is well known in 
probabilistic number theory and has a long history, see, for instance, Eqs.~(2.5)--(2.7) in \cite{Kubilius:1964} and \cite{Bill:74}. Note that in this case, the components $X_1^{(n)},\ldots,X_d^{(n)}$ are independent and $X_j^{(n)}$ has the uniform distribution on $\{1,\ldots,n_j\}$, for every $j=1,\ldots,d$. 
\end{rem}

\subsection{Limit theorems for \texorpdfstring{$F$}{F}}
We start with finding conditions ensuring a.s.~finiteness of $F_{\infty}$ in \eqref{eq:main_convergence_F}. Recall that we assume $d\geq 2$. According to Eq.~(20) in \cite{BosMarRas:2019} (or just by an appeal to the Borel-Cantelli lemma), we have
$$
\sum_{p\in\mathcal{P}}\1_{\{\sum_{k=1}^{d}\mathcal{G}_k(p)\geq 2\}}<\infty\quad \text{a.s.}
$$
Furthermore, because $F$ is multiplicative, $F(1,1,\ldots,1)=1$. Thus, a.s.~finiteness of $F_{\infty}$ is equivalent to the a.s.~convergence of the product
$$
\widehat{F}_{\infty}:=\prod_{p\in\mathcal{P}\;:\;\sum_{k=1}^{d}\mathcal{G}_k(p)=1}F(p^{\mathcal{G}_1(p)},\ldots,p^{\mathcal{G}_d(p)}).
$$
For $i=1,\ldots,d$, put
$$
F_i(x):=\log F(1,\ldots,1,x,1,\ldots,1),
$$
where $x\in\N$ on the right-hand side is on the $i$-th position and $\log$ is the principal branch of the logarithm (a branch which satisfies $\log(1)=0$ and has a branch cut along $(-\infty, 0]$). We assume that for all $i=1,\ldots,d$, there are only finitely many $p\in\mathcal{P}$ such that $F(1,\ldots,1,p,1,\ldots,1)$ falls inside the branch cut. Otherwise, we stipulate that the series diverges. Thus, the a.s.~convergence of $\widehat{F}_{\infty}$, hence of $F_{\infty}$, is equivalent to the a.s.~convergence of the series
\begin{equation}\label{eq:product_to_series}
\sum_{p\in\mathcal{P}}\left(\sum_{i=1}^{d}F_i(p)\1_{\{\mathcal{G}_i(p)=1,\mathcal{G}_j(p)=0\text{ for }j\neq i\}}\right),
\end{equation}
comprised of independent random variables. An application of Kolmogorov's three series theorem immediately yields the following:
\begin{assertion}\label{prop:three_series}
The infinite product $F_{\infty}$ converges a.s.~if and only if the following series converge for every $A>0$:
\begin{equation}\label{eq:three_series}
\sum_{p\in\mathcal{P}}\frac{1}{p}\sum_{i=1}^{d}\1_{\{|F_i(p)|>A\}},\quad\sum_{p\in\mathcal{P}}\frac{1}{p}\sum_{i=1}^{d}F_i(p)\1_{\{|F_i(p)|\leq A\}},\quad
\sum_{p\in\mathcal{P}}\frac{1}{p}\sum_{i=1}^{d}|F_i(p)|^2\1_{\{|F_i(p)|\leq A\}}.
\end{equation}
\end{assertion}
It is clear that the convergence of the three series~\eqref{eq:three_series} is a necessary condition for~\eqref{eq:main_convergence_F}. Proving~\eqref{eq:main_convergence_F} under~\eqref{eq:three_series} alone seems to be a very difficult task, even for simple regions $\mathcal{D}_n$ as $\mathcal{R}_{n,\ldots,n}$.

In this paper, we restrict our attention to a subclass of multivariate multiplicative functions satisfying~\eqref{eq:three_series}. Namely, we shall assume that, for all $i=1,\ldots,d$,
\begin{equation}\label{eq:two_series}
\sum_{p\in\mathcal{P}}\frac{1}{p}\1_{\{|F_i(p)|>A\}}<\infty\quad\text{and}\quad
\sum_{p\in\mathcal{P}}\frac{1}{p}|F_i(p)|\1_{\{|F_i(p)|\leq A\}}<\infty.
\end{equation}
It is obvious that~\eqref{eq:two_series} implies~\eqref{eq:three_series}. The difference between conditions~\eqref{eq:three_series} and~\eqref{eq:two_series} is that~\eqref{eq:two_series} is necessary and sufficient for the a.s.~{\it absolute} convergence of the series~\eqref{eq:product_to_series}, whereas under~\eqref{eq:three_series} the a.s.~convergence of the series~\eqref{eq:product_to_series} is, in general, only conditional.

In order to prove~\eqref{eq:main_convergence_F} under~\eqref{eq:two_series}, we shall impose a mild additional assumption on $\mathcal{D}_n$. For $i=1,\ldots,d$ and $a\in\N$, put
$$
\mathbb{Z}_i(a):=\{(x_1,\ldots,x_d)\in\mathbb{Z}^d\;:\; x_i\text{ is divisible by } a\}.
$$
As we shall see below in Lemma~\ref{lem:uniformity}, solely under assumption~\eqref{eq:van_Hove}, one has
\begin{equation}\label{eq:pairwise_convergence_van_hove}
\lim_{n\to\infty}\frac{\#(\mathcal{D}_n\cap \mathbb{Z}_i(a)\cap \mathbb{Z}_j(b))}{\#\mathcal{D}_n}=\frac{1}{ab},
\end{equation}
for every fixed $a,b\in\N$ and $i,j=1,\ldots,d$, $i\neq j$. However, we shall need a further assumption that refines the above limit relation, providing a kind of uniformity in~\eqref{eq:pairwise_convergence_van_hove}. Namely, we assume that there exists $K>0$ such that for all $i,j=1,\ldots,d$, $i\neq j$, $a,b\in\N$ and $n\in\N$,
\begin{equation}
\label{eq:additional_assumption_new01}
   \frac{\#(\mathcal{D}_n\cap \mathbb{Z}_i(a)\cap \mathbb{Z}_j(b))}{\#\mathcal{D}_n}\leq\frac{K}{ab}.
\end{equation}

Recall that $(X_1^{(n)},\ldots,X_d^{(n)})$ is a random vector picked uniformly at random from $\mathcal{D}_n$. Below is our main result.
\begin{thm}\label{thm:main_F_positive}
Assume that $F:\N^d\to \mathbb{C}$ is a multiplicative arithmetic function such that conditions~\eqref{eq:two_series} hold. Let $\mathcal{D}_n$, $n\in\N$, be a sequence of subsets of $\N^d$ such that~\eqref{eq:van_Hove} and~\eqref{eq:additional_assumption_new01} hold. Then
$$
F(X_1^{(n)},\ldots,X_d^{(n)})\todistrn \prod_{p\in\mathcal{P}}F(p^{\mathcal{G}_1(p)},\ldots,p^{\mathcal{G}_d(p)}).
$$
\end{thm}

Examples of integer domains satisfying~\eqref{eq:van_Hove} and~\eqref{eq:additional_assumption_new01} will be presented in Section~\ref{sec:examples_domains}.

The following functions $F$
$$
\N^d\ni (x_1,\ldots,x_d)\mapsto \GCD(x_1,\ldots,x_d)\quad\text{and}\quad \N^d\ni (x_1,\ldots,x_d)\mapsto \frac{\LCM(x_1,\ldots,x_d)}{x_1 \cdots x_d}
$$
are multiplicative and satisfy $F_i(x)\equiv 0$ for every $i=1,\ldots,d$. Thus, Theorem~\ref{thm:main_F_positive} is applicable, leading to the following corollaries.

\begin{cor}\label{thm:gcd}
Assume that~\eqref{eq:van_Hove} and~\eqref{eq:additional_assumption_new01} hold. Then
$$
\GCD(X_1^{(n)},\ldots,X_d^{(n)})~\todistrn~\prod_{p\in\mathcal{P}}p^{\min_{k=1,\ldots,d}\mathcal{G}_k(p)}.
$$
The limiting random variable has the following distribution
\begin{equation}\label{eq:gcd_limit}
\mathbb{P}\left\{\prod_{p\in\mathcal{P}}p^{\min_{k=1,\ldots,d}\mathcal{G}_k(p)}=j\right\}=\frac{1}{\zeta(d)}\frac{1}{j^d},\quad j\in\N,
\end{equation}
where $\zeta$ is the Riemann zeta function.
\end{cor}

\begin{cor}\label{thm:lcm}
Assume that~\eqref{eq:van_Hove} and~\eqref{eq:additional_assumption_new01} hold. Then
$$
\frac{\LCM(X_1^{(n)},\ldots,X_d^{(n)})}{X_1^{(n)}\cdots X_d^{(n)}}~\todistrn~\prod_{p\in\mathcal{P}}p^{\max_{k=1,\ldots,d}\mathcal{G}_k(p)-\sum_{k=1}^{d}\mathcal{G}_k(p)}.
$$
\end{cor}

\begin{rem}[Bibliographic comments] Below is a comparison of our results with the existing ones.

\noindent
{\sc Case $\mathcal{D}_n=\mathcal{R}_{n,\ldots,n}$.} In this case Corollaries~\ref{thm:gcd} and~\ref{thm:lcm} are known, with Corollary~\ref{thm:gcd} having a long history. The fact that two independent random integers picked uniformly at random from $\{1,\ldots,n\}$ are asymptotically co-prime with probability $1/\zeta(2)=6/\pi^2$, that is
$$
\lim_{n\to\infty}\mathbb{P}\{\GCD(X_1^{(n)},X_2^{(n)})=1\}=\frac{6}{\pi^2}
$$
goes back to Dirichlet \cite{Dirichlet}, and generalizations of this relation to $d>2$ integers are due to Ces\`{a}ro \cite{Cesaro1,Cesaro2}. To the best of our knowledge, Corollary~\ref{thm:gcd} is due to Christopher~\cite{Christopher:1956}, see also~\cite{Cohen:1960}. Formula~\eqref{eq:gcd_limit} follows from the following chain of equalities. For $s<d-1$, by Euler’s product formula 
$$
\E\left(\prod_{p\in\mathcal{P}}p^{\min_{k=1,\ldots,d}\mathcal{G}_k(p)}\right)^s=\prod_{p\in\mathcal{P}}\E p^{s\min_{k=1,\ldots,d}\mathcal{G}_k(p)}=\prod_{p\in\mathcal{P}}\left(1-\frac{1}{p^d}\right)\frac{1}{1-p^{s-d}}=\frac{\zeta(d-s)}{\zeta(d)}=\frac{1}{\zeta(d)}\sum_{j=1}^{d}\frac{j^{s}}{j^d}.
$$
Corollary~\ref{thm:lcm} can be extracted from Theorem~2.1 in \cite{HilberdinkToth:2016} and is given explicitly in Remark~2.4 in \cite{BosMarRas:2019}. Further pointers to literature related to Corollaries~\ref{thm:gcd} and~\ref{thm:lcm} in case $\mathcal{D}_n=\mathcal{R}_{n,\ldots,n}$ can be found in the introduction~\cite{BosMarRas:2019} and in the survey~\cite{Fernandez+Fernandez:2021}. In~\cite{BosMarRas:2019} a version of Theorem~\ref{thm:main_F_positive} was proved assuming that $F(x_1,\ldots,x_d)=G(\LCM(x_1,\ldots,x_d))$ for some univariate multiplicative function $G:\N\to\mathbb{C}$. Asymptotics of moments accompanying the aforementioned distributional convergences have been derived in~\cite{HilberdinkToth:2016,Toth:2014,Ushiroya:2012}. 

\noindent
{\sc Case $\mathcal{D}_n=\mathcal{H}_n$ (and more general hyperbolic regions, see Example~\ref{ex:further_hyperbolic} below).} In  this case, Corollaries~\ref{thm:gcd} and~\ref{thm:lcm} can be found in Theorems~3.5 and 3.7 in~\cite{IksMarRas:2022}. The  corresponding asymptotics of moments has been derived in~\cite{HeyToth:2021,HeyToth:2022-1}.

\noindent
{\sc Case $\mathcal{D}_n=\mathcal{S}_n$.} The distributional convergence is completely new. The asymptotics of moments has been analyzed in~\cite{HeyToth:2022-2}.
\end{rem}

\section{Proof of the main results}\label{sec:proofs}
\subsection{Proof of Theorem~\ref{prop:geometric}}
We first need an auxiliary lemma.

\begin{lemma}\label{lem:uniformity}
Fix $m_1,\ldots,m_d\in\N$ and $j_k\in\{0,\ldots,m_k-1\}$, $k=1,\ldots,d$. Put
$$
\mathcal{D}_{n}^{(j_1,m_1,\ldots,j_d,m_d)}:=\{(i_1,\ldots,i_d)\in\mathcal{D}_n: i_k\equiv j_k\ (\mod m_k)\text{ for all }k=1,\ldots,d\}.
$$
If~\eqref{eq:van_Hove} holds, then
\begin{equation}\label{eq:vanHove_implies uniformity}
\lim_{n\to\infty}\frac{\#\mathcal{D}_{n}^{(j_1,m_1,\ldots,j_d,m_d)}}{\#\mathcal{D}_n}=\frac{1}{m_1\cdots m_d}.
\end{equation}
\end{lemma}
\begin{proof}
Note that
\begin{equation}\label{eq:d_n_decomposition1}
\mathcal{D}_n=\bigcup_{j_1=0}^{m_1-1}\cdots\bigcup_{j_d=0}^{m_d-1}\mathcal{D}_{n}^{(j_1,m_1,\ldots,j_d,m_d)},
\end{equation}
and the sets on the right-hand side are pairwise disjoint. Furthermore,
\begin{multline*}
\mathcal{D}_{n}^{(j_1,m_1,\ldots,j_d,m_d)}=\mathcal{D}_n\cap (j_1+m_1\mathbb{Z},\ldots,j_d+m_d\mathbb{Z})=(j_1,\ldots,j_d)+(\mathcal{D}_n-(j_1,\ldots,j_d))\cap (m_1\mathbb{Z},\ldots,m_d\mathbb{Z}).
\end{multline*}
Thus,
\begin{align*}
&\hspace{-1cm}\left|\#\mathcal{D}_{n}^{(0,m_1,\ldots,0,m_d)}-\#\mathcal{D}_{n}^{(j_1,m_1,\ldots,j_d,m_d)}\right|\\
&=\left|\#(\mathcal{D}_{n}\cap (m_1\mathbb{Z},\ldots,m_d\mathbb{Z}))-\#((\mathcal{D}_{n}-(j_1,\ldots,j_d))\cap (m_1\mathbb{Z},\ldots,m_d\mathbb{Z}))\right|\\
&\leq \# \left((\mathcal{D}_{n}\cap (m_1\mathbb{Z},\ldots,m_d\mathbb{Z}))\Delta ((\mathcal{D}_{n}-(j_1,\ldots,j_d))\cap (m_1\mathbb{Z},\ldots,m_d\mathbb{Z}))\right)\\
&\leq \# \left(\mathcal{D}_{n}\Delta (\mathcal{D}_{n}-(j_1,\ldots,j_d))\right),
\end{align*}
and we have proved that (with $\delta_n$ introduced in \eqref{eq:delta})
\begin{equation}\label{eq:d_n_minus_d_j}
\left|\#\mathcal{D}_{n}^{(0,m_1,\ldots,0,m_d)}-\#\mathcal{D}_{n}^{(j_1,m_1,\ldots,j_d,m_d)}\right|\leq \delta_n(-(j_1,\ldots,j_d)).
\end{equation}
Plugging this into~\eqref{eq:d_n_decomposition1} yields
$$
\left|\#\mathcal{D}_n-m_1\cdots m_d \#\mathcal{D}_{n}^{(0,m_1,\ldots,0,m_d)}\right|\leq \sum_{j_1=0}^{m_1-1}\cdots\sum_{j_d=0}^{m_d-1}\delta_n(-(j_1,\ldots,j_d)).
$$
Dividing both sides by $\#\mathcal{D}_n$ and sending $n\to\infty$ implies~\eqref{eq:vanHove_implies uniformity} for $j_1=\cdots=j_d=0$. Using the estimate~\eqref{eq:d_n_minus_d_j}, we obtain~\eqref{eq:vanHove_implies uniformity} for arbitrary $j_1,\ldots,j_d$.
\end{proof}

\begin{proof}[Proof of Theorem~\ref{prop:geometric}]
Fix pairwise distinct prime numbers $p_1,\ldots,p_m\in\mathcal{P}$, nonnegative integers $j_{k,t}$, $k=1,\ldots,d$, $t=1,\ldots,m$, and write
\begin{align*}
&\hspace{-0.3cm}\mmp\{\lambda_{p_t}(X_k^{(n)})\geq j_{k,t}\text{ for all }k=1,\ldots,d\text{ and }t=1,\ldots,m\}\\
&=\mmp\{X_k^{(n)}\text{ is divisible by }p_t^{j_{k,t}}\text{ for all }k=1,\ldots,d\text{ and }t=1,\ldots,m\}\\
&=\mmp\{X_k^{(n)}\text{ is divisible by }\prod_{t=1}^{m}p_t^{j_{k,t}}=:\mu_k\text{ for all }k=1,\ldots,d\}\\
&=\frac{1}{\#\mathcal{D}_n}\sum_{i_1=1}^{\infty}\cdots\sum_{i_d=1}^{\infty}\1\left\{(i_1,\ldots,i_d)\in\mathcal{D}_n:i_k\equiv 0\ (\mod \mu_k),k=1,\ldots,d \right\}.
\end{align*}
By Lemma~\ref{lem:uniformity} applied with $m_k=\mu_k$ and $j_k=0$, $k=1,\ldots,d$, we see that the right-hand side converges to
$(\mu_1\cdots \mu_d)^{-1}$ as $n\to\infty$. It remains to note that
$$
\frac{1}{\mu_1\cdots \mu_d}=\prod_{k=1}^{d}\prod_{t=1}^{m}\frac{1}{p_t^{j_{k,t}}}=\mmp\{\mathcal{G}_k(p_t)\geq j_{k,t}\text{ for all }k=1,\ldots,d\text{ and }t=1,\ldots,m\}.
$$
The proof of Theorem~\ref{prop:geometric} is complete.
\end{proof}

\subsection{Proof of Theorem~\ref{thm:main_F_positive}}

Fix a large positive constant $M$ and note that
\begin{multline*}
F(X_1^{(n)},\ldots,X_d^{(n)})=\prod_{p\in\mathcal{P}}F(p^{\lambda_p(X_1^{(n)})},\ldots,p^{\lambda_p(X_d^{(n)})})\\
=\left(\prod_{p\in\mathcal{P},p\leq M}F(p^{\lambda_p(X_1^{(n)})},\ldots,p^{\lambda_p(X_d^{(n)})})\right)\left(\prod_{p\in\mathcal{P},p>M}F(p^{\lambda_p(X_1^{(n)})},\ldots,p^{\lambda_p(X_d^{(n)})})\right)=:Y_1(M,n)Y_2(M,n).
\end{multline*}
By Theorem~\ref{prop:geometric}, one has
$$
Y_1(M,n)~\todistrn~\prod_{p\in\mathcal{P},p\leq M}F(p^{\mathcal{G}_1(p)},\ldots,p^{\mathcal{G}_d(p)}).
$$
Furthermore, the right-hand side of the latter converges a.s.~to
$F_{\infty}$ as $M\to\infty$, which is a.s.~finite. According to Theorem 3.2 in~\cite{Billingsley:1968}, it remains to check that for every fixed $\varepsilon>0$,
\begin{equation}\label{eq:bill_gcd}
\lim_{M\to\infty}\limsup_{n\to\infty}\mmp\left\{\left| Y_2(M,n)-1\right|\geq \varepsilon \right\}=0.
\end{equation}
Note that
\begin{multline}\label{eq:bill_proof1}
\mmp\left\{\left| Y_2(M,n)-1\right|\geq \varepsilon \right\}\leq \mmp\left\{\text{for all }p\in\mathcal{P},p>M, \sum_{i=1}^{d}\lambda_p(X_i^{(n)})\leq 1,\left|Y_2(M,n)-1\right|\geq \varepsilon \right\}\\
+\mmp\left\{\text{for some }p\in\mathcal{P},p>M, \sum_{i=1}^{d}\lambda_p(X_i^{(n)})\geq 2\right\}.
\end{multline}
The second term in \eqref{eq:bill_proof1} can be estimated as follows:
\begin{align*}
&\hspace{-0.2cm}\mmp\{\text{for some }p\in\mathcal{P},p>M, \sum_{i=1}^{d}\lambda_p(X_i^{(n)})\geq 2\}\\
&\leq \mmp\{\text{there exist }p\in\mathcal{P},p>M\text{ and }i=1,\ldots,d\text{ such that } \lambda_p(X_i^{(n)})\geq 2\}\\
&\hspace{4mm}+\mmp\{\text{there exist }p\in\mathcal{P},p>M\text{ and }i,j=1,\ldots,d, i\neq j\text{ such that } \lambda_p(X_i^{(n)})\geq 1,\lambda_p(X_j^{(n)})\geq 1\}\\
&= \mmp\{\text{there exist }p\in\mathcal{P},p>M\text{ and }i=1,\ldots,d\text{ such that } p^2 \text{ divides } X_i^{(n)}\}\\
&\hspace{4mm}+\mmp\{\text{there exist }p\in\mathcal{P},p>M\text{ and }i,j=1,\ldots,d, i\neq j\text{ such that } p \text{ divides } X_i^{(n)}\text{ and }X_j^{(n)}\}\\
&\leq \sum_{i=1}^{d}\sum_{p\in\mathcal{P},p>M}\mmp\{p^2 \text{ divides } X_i^{(n)}\}+\sum_{i,j=1,i\neq j}^{d}\sum_{p\in\mathcal{P},p>M}\mmp\{p \text{ divides } X_i^{(n)}\text{ and }X_j^{(n)}\}\\
&=\sum_{i=1}^{d}\sum_{p\in\mathcal{P},p>M}\frac{\#(\mathcal{D}_n\cap \mathbb{Z}_{i}(p^2))}{\#\mathcal{D}_n}+\sum_{i,j=1,i\neq j}^{d}\sum_{p\in\mathcal{P},p>M}\frac{\#(\mathcal{D}_n\cap \mathbb{Z}_{i}(p)\cap \mathbb{Z}_j(p))}{\#\mathcal{D}_n}.
\end{align*}
The double limit ($n\to\infty$, $M\to\infty$) of the first term is equal to zero by an appeal to~\eqref{eq:additional_assumption_new01} with $a=p^2$ and $b=1$, since
$$
\lim_{M\to\infty}\sum_{p\in\mathcal{P},p>M}\frac{1}{p^2}=0.
$$
Similarly, the double limit of the second term is equal to zero by an appeal to~\eqref{eq:additional_assumption_new01} with $a=b=p$.

In order to deal with the first summand in~\eqref{eq:bill_proof1}, we first observe that on the event
$$
\left\{\text{for all }p\in\mathcal{P},p>M,\sum_{i=1}^{d}\lambda_p(X_i^{(n)})\leq 1\right\},
$$
we may pass to the logarithm of $Y_2(M,n)$. Thus, it suffices to prove that, for every $\varepsilon>0$,
$$
\lim_{M\to\infty}\limsup_{n\to\infty}\mmp\left\{\text{for all }p\in\mathcal{P},p>M, \sum_{i=1}^{d}\lambda_p(X_i^{(n)})\leq 1,\left|\sum_{p\in\mathcal{P},p>M}\log F(p^{\lambda_p(X_1^{(n)})},\ldots,p^{\lambda_p(X_d^{(n)})})\right|\geq \varepsilon \right\}=0.
$$
Introduce, for $n\in\N$, $ i=1,\ldots,d$ and $p\in\mathcal{P}$, the events
$$
C_{n,i,p}:=\{\lambda_p(X_i^{(n)})=1,\lambda_p(X_j^{(n)})=0,j\neq i\},
$$
and note that $C_{n,i,p}\cap C_{n,j,p}=\varnothing$ as soon as $i\neq j$. On the event $C_{n,i,p}$, we have 
$$
\log F(p^{\lambda_p(X_1^{(n)})},\ldots,p^{\lambda_p(X_d^{(n)})})=F_i(p)
$$ 
and, therefore, it suffices to show that, for every fixed $\varepsilon>0$,
\begin{equation}\label{eq:main_thm_F_proof1}
\lim_{M\to\infty}\limsup_{n\to\infty}\mmp\left\{\left|\sum_{p\in\mathcal{P},p>M}\sum_{i=1}^{d}F_i(p)\1_{C_{n,i,p}}\right|\geq \varepsilon \right\}=0.
\end{equation}
Fix some $A>0$ and note that, for every $\varepsilon>0$,
\begin{align*}
&\hspace{-1cm}\mmp\left\{\left|\sum_{p\in\mathcal{P},p>M}\sum_{i=1}^{d}F_i(p)\1_{\{|F_i(p)|>A, C_{n,i,p}\}}\right|\geq \varepsilon \right\}\\
&\leq
\mmp\{\text{for some }p\in\mathcal{P}\text{ and }i=1,\ldots,d,\ |F_i(p)|>A\text{ and }C_{n,i,p}\text{ holds}\}\\
&\leq\sum_{p\in\mathcal{P},p>M}\sum_{i=1}^{d}\1_{\{|F_i(p)|>A\}}\mmp\{C_{n,i,p}\}\leq \sum_{p\in\mathcal{P},p>M}\sum_{i=1}^{d}\1_{\{|F_i(p)|>A\}}\mmp\{\lambda_p(X_i^{(n)})\geq 1\}\\
&=\sum_{p\in\mathcal{P},p>M}\sum_{i=1}^{d}\1_{\{|F_i(p)|>A\}}\frac{\#(\mathcal{D}_n\cap \mathbb{Z}_{i}(p))}{\#\mathcal{D}_n}\leq K\sum_{p\in\mathcal{P},p>M}\frac{1}{p}\sum_{i=1}^{d}\1_{\{|F_i(p)|>A\}},
\end{align*}
where we used~\eqref{eq:additional_assumption_new01} with $a=p$ and $b=1$ for the last passage. The right-hand side converges to zero as $M\to\infty$, in view of the first relation in~\eqref{eq:three_series}. So, in order to prove~\eqref{eq:main_thm_F_proof1}, we need to check that
\begin{equation}\label{eq:main_thm_F_proof2}
\lim_{M\to\infty}\limsup_{n\to\infty}\mmp\left\{\left|\sum_{p\in\mathcal{P},p>M}\sum_{i=1}^{d}F_i(p)\1_{\{|F_i(p)|\leq A,C_{n,i,p}\}}\right|\geq \varepsilon \right\}=0.
\end{equation}

This is accomplished by an appeal to Markov's inequality as follows:
\begin{align*}	
\mmp\left\{\left|\sum_{p\in\mathcal{P},p>M}\sum_{i=1}^{d}F_i(p)\1_{\{|F_i(p)|\leq A,C_{n,i,p}\}}\right|\geq \varepsilon \right\}&\leq \frac{1}{\varepsilon}\sum_{p\in\mathcal{P},p>M}\sum_{i=1}^{d}|F_i(p)|\1_{\{|F_i(p)|\leq A\}}\mmp\{C_{n,i,p}\}\\
&\leq \frac{1}{\varepsilon}\sum_{p\in\mathcal{P},p>M}\sum_{i=1}^{d}|F_i(p)|\1_{\{|F_i(p)|\leq A\}}\mmp\{\lambda_p(X_i^{(n)})\geq 1\}\\
&= \frac{1}{\varepsilon}\sum_{p\in\mathcal{P},p>M}\sum_{i=1}^{d}F_i(p)\1_{\{|F_i(p)|\leq A\}}\frac{\#(\mathcal{D}_n\cap \mathbb{Z}_{i}(p))}{\#\mathcal{D}_n}\\
&\leq \frac{K}{\varepsilon}\sum_{p\in\mathcal{P},p>M}\frac{1}{p}\sum_{i=1}^{d}F_i(p)\1_{\{|F_i(p)|\leq A\}},
\end{align*}
where we have utilized~\eqref{eq:additional_assumption_new01} with $a=p$ and $b=1$ for the last inequality. The proof of Theorem~\ref{thm:main_F_positive} is complete, since the right-hand side converges to zero, as $M\to\infty$, by the second relation in~\eqref{eq:two_series}.

\section{Examples of suitable integer domains}\label{sec:examples_domains}

In this section we provide a series of examples of domains $\mathcal{D}_n$ that satisfy~\eqref{eq:van_Hove} and~\eqref{eq:additional_assumption_new01}. In particular, we show  that $\mathcal{R}_{n_1,n_2,\ldots,n_d}$ in \eqref{eq:rectangular_domain_def}, $\mathcal{S}_n$ and $\mathcal{H}_n$ mentioned in the introduction, are all admissible. Thus, under assumption~\eqref{eq:two_series} on $F$, the distributional convergence~\eqref{eq:main_convergence_F} holds true for all domains listed below.

\subsection{Sublevels of monotone functions}
\begin{assertion}\label{prop:example1}
Assume that $f:[1,\infty)^d \to \mathbb{R}$ is a coordinate-wise nondecreasing function such that, for every $j=1,\ldots,d$,
$$
\lim_{x_j\to \infty}f(x_1,\ldots,x_d)=\infty,
$$
provided $x_i\geq 1$, $i\neq j$, are fixed. Put
$$
\mathcal{D}_n:=\mathcal{D}_n^f=\{(x_1,\ldots,x_d)\in\N^d: f(x_1,\ldots,x_d)\leq n\}
$$
and
$$
\mathcal{D}_{n,i}:=\mathcal{D}_{n,i}^f=\{(x_1,\ldots,x_{i-1},x_{i+1},\ldots,x_d)\in\N^{d-1}: f(x_1,\ldots,x_{i-1},1,x_{i+1},\ldots,x_d)\leq n\},
$$
for $i=1,\ldots,d$. If, for every $i=1,\ldots,d$,
\begin{equation}\label{eq:boundary_negligible}
\lim_{n\to\infty}\frac{\#\mathcal{D}_{n,i}}{\#\mathcal{D}_{n}}=0,
\end{equation}
then the sequence $\mathcal{D}_n$, $n\in\mathbb{N}$, satisfies~\eqref{eq:van_Hove} and~\eqref{eq:additional_assumption_new01}.
\end{assertion}
\begin{proof}
Let us first verify~\eqref{eq:van_Hove}. According to Proposition~\ref{prop:def_equivalence} in Appendix~\ref{sec:appendix}, it is sufficient to check~\eqref{eq:van_Hove} for $c=e_i$, $i=1,\ldots,d$, where $e_1,\ldots,e_d$ denotes the standard basis of $\mathbb{R}^d$. Note that $\mathcal{D}_{n}\setminus (\mathcal{D}_n+e_i)=\mathcal{D}_{n,i}$. Thus,~\eqref{eq:boundary_negligible} yields that for $i=1,\ldots,d$,
$$
\lim_{n\to\infty}\frac{\#(\mathcal{D}_{n}\setminus (\mathcal{D}_n+e_i))}{\#\mathcal{D}_n}=0.
$$
It remains to check that for $i=1,\ldots,d$,
\begin{equation}\label{eq:example1_proof1}
\lim_{n\to\infty}\frac{\#((\mathcal{D}_{n}+e_i)\setminus \mathcal{D}_n)}{\#\mathcal{D}_n}=0.
\end{equation}
Without loss of generality, we shall do this for $i=1$. Note that
$$
(\mathcal{D}_{n}+e_1)\setminus \mathcal{D}_n=\{(x_1,\ldots,x_d)\in\N^d:x_1\geq 2,f(x_1-1,x_2,\ldots,x_d)\leq n,f(x_1,\ldots,x_d)>n\}.
$$
For every fixed collection $(x_2,\ldots,x_d)\in\N^{d-1}$ and $n\in\N$, there exists at most one $x_1\geq 2$, $x_1\in\N$, such that
$$
f(x_1-1,x_2,\ldots,x_d)\leq n\quad\text{and}\quad f(x_1,\ldots,x_d)>n,
$$
since $f$ is monotone in $x_1$. Therefore,
\begin{align*}
\#((\mathcal{D}_{n}+e_1)\setminus \mathcal{D}_n)&=\sum_{x_2=1}^{\infty}\cdots\sum_{x_d=1}^{\infty}\1_{\{\text{there exists }x_1\geq 2\text{ such that }f(x_1-1,x_2,\ldots,x_d)\leq n,f(x_1,\ldots,x_d)>n\}}\\
&\leq \sum_{x_2=1}^{\infty}\cdots\sum_{x_d=1}^{\infty}\1_{\{\text{there exists }x_1\geq 2\text{ such that }f(x_1-1,x_2,\ldots,x_d)\leq n\}}=\sum_{x_2=1}^{\infty}\cdots\sum_{x_d=1}^{\infty}\1_{\{f(1,x_2,\ldots,x_d)\leq n\}}\\
&=\#\mathcal{D}_{n,1}.
\end{align*}
This proves~\eqref{eq:example1_proof1} for $i=1$.

We shall now prove that~\eqref{eq:additional_assumption_new01} holds, for all $i,j=1,\ldots,d$, with $K=1$. For notational simplicity, we shall do this only for $i=1$ and $j=2$. The monotonicity of $f$ implies that, for all $a,b\in\N$,
\begin{align*}
\#\mathcal{D}_n&=\sum_{j=0}^{a-1}\sum_{k=0}^{b-1}\left(\sum_{x_1=1}^{\infty}\sum_{x_2=1}^{\infty}\cdots\sum_{x_d=1}^{\infty}\1_{\{f(ax_1-j,b x_2-k,x_3,\ldots,x_d)\leq n\}}\right)\\
&\geq ab \sum_{x_1=1}^{\infty}\sum_{x_2=1}^{\infty}\cdots\sum_{x_d=1}^{\infty}\1_{\{f(a x_1,b x_2,x_3,\ldots,x_d)\leq n\}}\\
&=ab\#(\mathcal{D}_n\cap \mathbb{Z}_{1}(a)\cap\mathbb{Z}_2(b)).
\end{align*}
The proof of Proposition~\ref{prop:example1} is complete.
\end{proof}

Proposition~\ref{prop:example1} yields the following explicit examples.

\begin{example}[Rectangular domains]
Let $f_1,\ldots,f_d:[1,\infty)\to [1,\infty)$ be strictly increasing continuous functions. Putting
$f(x_1,\ldots,x_d):=\max(f_1^{-1}(x_1),\ldots,f_d^{-1}(x_d))$, we obtain
\begin{equation*}
   \mathcal{D}_n=\mathcal{R}_{f_1(n),\ldots,f_d(n)}=([1,f_1(n)]\times\cdots \times[1,f_d(n)])\cap \N^d.
\end{equation*}
Condition~\eqref{eq:boundary_negligible} is fulfilled if
$\lim_{x\to\infty}f_i(x)=\infty$, for every $i=1,\ldots,d$.
\end{example}

\begin{example}[Tetrahedral domains]
Let $a_1,\ldots,a_d>0$ be fixed positive real numbers. The sequence of tetrahedral sets
$$
\mathcal{D}_n=\mathcal{T}_n:=\{(x_1,\ldots,x_d)\in\N^d:a_1x_1+\cdots+a_dx_d\leq n\}
$$
satisfies~\eqref{eq:van_Hove} and~\eqref{eq:additional_assumption_new01}. Indeed,
$$
\#\mathcal{T}_n~\sim~\frac{1}{d! a_1 \cdots a_d}n^{d},\quad n\to\infty,
$$
whereas, for $i=1,\ldots,d$,
$$
\#\mathcal{T}_{n,i}~\sim~\frac{a_i}{(d-1)!a_1\cdots a_d} n^{d-1},\quad n\to\infty.
$$
Thus, Proposition~\ref{prop:example1} is applicable.

\end{example}

\begin{example}[Hyperbolic domains]
\label{ex:l=d}
Let $f(x_1,\ldots,x_d)=x_1 \cdots x_d$. Then the sequence of sets
$$
\mathcal{D}_n=\mathcal{H}_n:=\{(x_1,\ldots,x_d)\in\N^d: x_1\cdots x_d\leq n\}
$$
satisfies~\eqref{eq:van_Hove} and~\eqref{eq:additional_assumption_new01}. Indeed, according to Proposition 4.1 in~\cite{IksMarRas:2022}.
$$
\#\mathcal{D}_n~\sim~\frac{n\log^{d-1}n}{(d-1)!},\quad n\to\infty,
$$
and, for every $i=1,\ldots,d$,
$$
\#\mathcal{D}_{n,i}~\sim~\frac{n\log^{d-2}n}{(d-2)!},\quad n\to\infty.
$$
Thus, Proposition~\ref{prop:example1} is applicable.
\end{example}

\begin{example}[Further hyperbolic domains]\label{ex:further_hyperbolic}
Fix $2\leq \ell\leq d$. Define the $\ell$-th standard symmetric polynomial in $d$ variables by
$$f(x_1,\ldots,x_d)=P_\ell(x_1,\ldots, x_d) := \sum_{1\leq i_1<\cdots <i_\ell\leq d} x_{i_1}\cdots x_{i_\ell}.
$$ 
The associated domain is
$$
\mathcal{D}_n=\mathcal{H}_{\ell,d}(n):=\{(x_1,\ldots,x_d)\in\N^d: P_\ell(x_1,\ldots, x_d)\leq n\}.
$$
Example~\ref{ex:l=d} corresponds to the particular case $\ell=d$. If now $2\leq \ell<d$, then Proposition 4.4 in~\cite{IksMarRas:2022} entails that 
$\#\mathcal{D}_n\sim C(d,\ell)n^{d/\ell}$, for some positive constant $C(d,\ell)>0$. Furthermore, by symmetry $\# \mathcal{D}_{n,i}=\#\mathcal{D}_{n,1}$ for all $i=1,\ldots,d$, and
\begin{align*}
\mathcal{D}_{n,1}&=\{(x_2,\ldots,x_d)\in\N^d:P_\ell(1, x_2,\ldots, x_d)\leq n\}\\
&=\{(x_2,\ldots,x_d)\in\N^d:P_\ell(x_2,\ldots, x_d)+P_{\ell-1}(x_2,\ldots, x_d)\leq n\}\\
&\subset \{(x_2,\ldots,x_d)\in\N^d: P_\ell(x_2,\ldots, x_d)\leq n\}=\mathcal{H}_{\ell,d-1}(n).
\end{align*}
Thus, $\mathcal{D}_{n,1}\subseteq \mathcal{H}_{\ell,d-1}(n)$ and thereupon $\#\mathcal{D}_{n,1}\leq \#\mathcal{H}_{\ell,d-1}(n)$. If $\ell<d-1$, then
$$
\#\mathcal{H}_{\ell,d-1}(n)\sim C(d-1,\ell)n^{(d-1)/\ell},\quad n\to\infty,
$$
whereas if $\ell=d-1$,
$$
\#\mathcal{H}_{\ell,d-1}(n)=\#\mathcal{H}_{d-1,d-1}(n)\sim \frac{n\log^{d-2} n}{(d-2)!},\quad n\to\infty.
$$
In both cases $\lim_{n\to\infty}\#\mathcal{H}_{\ell,d-1}(n)/\#\mathcal{D}_n=0$. Summarizing, Proposition~\ref{prop:example1} is applicable to $\mathcal{D}_n=\mathcal{H}_{\ell,d}(n)$.
\end{example}

\subsection{Dilations of a convex body}

\begin{assertion}\label{prop:example2}
Let $\mathcal{D}\subset[0,\infty)^d$ be a compact convex set with nonempty interior and $a_n$, $n\in\N$, be a sequence of positive numbers such that $\lim_{n\to\infty}a_n=\infty$. Then, the following sequence of sets satisfies~\eqref{eq:van_Hove} and~\eqref{eq:additional_assumption_new01}:
$$
\mathcal{D}_n:=a_n\mathcal{D}\cap \N^d.
$$
\end{assertion}
\begin{proof}
For the proof of~\eqref{eq:van_Hove}, we shall use Proposition~\ref{prop:continuous-discrete}. Put
$$
V:=\mathcal{D}\cap (0,\infty)^d,\quad V_n:=a_n  V=a_n\mathcal{D}\cap (0,\infty)^d,
$$
and note that $\mathcal{D}_n = V_n\cap\mathbb{N}^d$. Let us check that~\eqref{eq:van_hove_continuous} holds for the sequence $V_n$. First of all, since $\mathcal{D}$ is compact, convex and has a non-empty interior, it holds 
$$
\mathcal{D}={\rm cl}({\rm Int}(\mathcal{D}))={\rm cl}({\rm Int}(\mathcal{D})\cap (0,\infty)^d)={\rm cl}({\rm Int}(\mathcal{D}\cap (0,\infty)^d))={\rm cl}(V),
$$
and, thereupon,
$$
\partial V= {\rm cl}(V)\setminus {\rm Int}(V)= {\rm cl}(V)\setminus {\rm Int}(\mathcal{D})=\mathcal{D}\setminus {\rm Int}(\mathcal{D})=\partial \mathcal{D}.
$$
Further, observe that ${\rm Vol}(V)>0$ and, denoting $B^{d}_{\varepsilon}(0)$ the ball $\{(x_1,\ldots,x_d)\in\mathbb R^d : x_1^2+\cdots +x_d^2 <\varepsilon\}$ and $A\oplus B:=\{x+y:x\in A,y\in B\}$ the Minkowski addition,
\begin{equation}\label{eq:example2_proof1}
\frac{{\rm Vol}(\partial V_n\oplus B^{d}_{\varepsilon}(0))}{{\rm Vol}(V_n)}=\frac{{\rm Vol}(a_n(\partial V\oplus B^{d}_{\varepsilon/a_n}(0)))}{{\rm Vol}(a_n V)}=\frac{{\rm Vol}(\partial V\oplus B^{d}_{\varepsilon/a_n}(0))}{{\rm Vol}(V)}=\frac{{\rm Vol}(\partial \mathcal{D}\oplus B^{d}_{\varepsilon/a_n}(0))}{{\rm Vol}(V)}.
\end{equation}
Since $\mathcal{D}$ is a compact convex set, its boundary $\partial \mathcal{D}$ is $(d-1)$-rectifiable subset of $\R^d$, that is, can be represented as the image of a Lipschitz function\footnote{As $h$ one can take, for example, the function $\partial B_R(0)\ni x\mapsto \pi_{\mathcal{D}}(x)$, where $R>0$ is such that $\mathcal{D}\subseteq B_R(0)$ and $\pi_{\mathcal{D}}(x)$ is a unique closest to $x$ point in $\mathcal{D}$ (metric projection on $\mathcal{D}$).} $h$ defined on a bounded subset of $\mathbb{R}^{d-1}$ and taking values in $\mathbb{R}^d$. Thus, by Theorem 3.2.39 in \cite{Federer_book},
$$
\lim_{n\to\infty}a_n{\rm Vol}(\partial \mathcal{D}\oplus B^{d}_{\varepsilon/a_n}(0))=2\varepsilon \mathcal{H}_{d-1}(\partial \mathcal{D})<\infty,
$$
where $\mathcal{H}_{d-1}$ is the $(d-1)$-dimensional Hausdorff  measure in $\mathbb{R}^d$. Summarizing, we have shown that the right-hand side of~\eqref{eq:example2_proof1} converges to zero as $n\to\infty$.

For the proof of~\eqref{eq:additional_assumption_new01}, we employ Proposition~\ref{prop:add_assump_suff2} from Appendix~\ref{sec:appendix}.

For $i=1,\ldots,d$, put
\begin{align*}
m_i(\mathcal{D})&:=\inf\{x_i\geq 0:(x_1,\ldots,x_i,\ldots,x_d)\in\mathcal{D}\},\\
M_i(\mathcal{D})&:=\sup\{x_i\geq 0:(x_1,\ldots,x_i,\ldots,x_d)\in\mathcal{D}\},
\end{align*}
and note that $0\leq m_i(\mathcal{D})<M_i(\mathcal{D})<\infty$. Here the second inequality is strict since $\mathcal{D}$ has a non-empty interior; the last inequality follows from the compactness of $\mathcal{D}$. Proposition~\ref{prop:add_assump_suff2} is applicable with the rectangle
$$
\Pi_n:=\left(\varprod_{i=1}^{d}\Big[\lfloor a_n m_i(\mathcal{D})\rfloor ,\lceil a_n M_i(\mathcal{D})\rceil\Big]\right)\bigcap \N^d.
$$
By construction
$$
a_n\mathcal{D} \subset a_n \left(\varprod_{i=1}^{d}\Big[m_i(\mathcal{D}), M_i(\mathcal{D})\Big]\right)\subset \left(\varprod_{i=1}^{d}\Big[\lfloor a_n m_i(\mathcal{D})\rfloor ,\lceil a_n M_i(\mathcal{D})\rceil\Big]\right).
$$
It remains to note that as $n\to\infty$,
\begin{equation}
\label{eq:example2_proof2}
   \#\Pi_n\sim a_n^d \prod_{i=1}^{d} (M_i(\mathcal{D})-m_i(\mathcal{D})),
\end{equation}
and also
\begin{equation}\label{eq:example2_proof3}
\liminf_{n\to\infty}\frac{\#\mathcal{D}_n}{a_n^d}>0,
\end{equation}
which is a consequence of the fact that $\mathcal{D}$ has a non-empty interior and, therefore, contains a small $d$-dimensional cube in the interior. Relations~\eqref{eq:example2_proof2} and~\eqref{eq:example2_proof3} imply
$$
\limsup_{n\to\infty}\frac{\#\Pi_n}{\#\mathcal{D}_n}<\infty.
$$
The proof of Proposition~\ref{prop:example2} is complete.
\end{proof}

\begin{example}[Spherical domains]
Put $\mathcal{B}:=\{(x_1,\ldots,x_d)\in [0,\infty)^d: x_1^2+\cdots+x_d^2\leq 1\}$. Then the sequence of discrete balls
$$
\mathcal{D}_n=\mathcal{S}_n:=\sqrt{n}\mathcal{B}\cap \N^d=\{(x_1,\ldots,x_d)\in\N^d: x_1^2+\cdots+x_d^2\leq n\}
$$
satisfies~\eqref{eq:van_Hove} and~\eqref{eq:additional_assumption_new01} by Proposition~\ref{prop:example2}.
\end{example}

Truncated cones, such as Weyl chambers, also satisfy~\eqref{eq:van_Hove} and~\eqref{eq:additional_assumption_new01}.
\begin{example}[Truncated Weyl chambers]
Let $\mathcal{A}:=\{(x_1,\ldots,x_d)\in [0,\infty): x_1\leq \cdots \leq x_d\leq 1\}$. Then the sequence of sets
$$
\mathcal{D}_n=\mathcal{A}_n:=n\mathcal{A}\cap \N^d=\{(x_1,\ldots,x_d)\in\N^d: x_1 \leq \cdots\leq x_d \leq n\}
$$
satisfies~\eqref{eq:van_Hove} and~\eqref{eq:additional_assumption_new01} by Proposition~\ref{prop:example2}.
\end{example}

\section{Set-theoretic operations preserving properties~\eqref{eq:van_Hove} and~\eqref{eq:additional_assumption_new01}}\label{sec:set-theoretic-operations}

In this section we discuss stability properties of sets satisfying~\eqref{eq:van_Hove} and~\eqref{eq:additional_assumption_new01} with respect to the standard set-theoretic operations.

First, immediately from the definitions, one obtains the following.
\begin{assertion}\label{prop:unions}
Let $\mathcal{D}^{(1)}_n$ and $\mathcal{D}^{(2)}_n$ be two sequences of sets satisfying~\eqref{eq:van_Hove} and~\eqref{eq:additional_assumption_new01}. Then the sequence~$\mathcal{D}_n:=\mathcal{D}_n^{(1)}\cup \mathcal{D}^{(2)}_n$ satisfies~\eqref{eq:van_Hove} and~\eqref{eq:additional_assumption_new01}.
\end{assertion}

As far as intersections and differences of sets are concerned, additional assumptions ensuring that the resulting sets are not small have to imposed. The following holds true.

\begin{assertion}\label{prop:differences}
Let $\mathcal{D}^{(1)}_n$ and $\mathcal{D}^{(2)}_n$ be two sequences of sets satisfying~\eqref{eq:van_Hove}. Suppose further that
\begin{equation}
\label{eq:difference_assump}
   \mathcal{D}^{(2)}_n\subset \mathcal{D}^{(1)}_n\quad\text{and}\quad \limsup_{n\to\infty}\frac{\#\mathcal{D}^{(2)}_n}{\#\mathcal{D}^{(1)}_n}\in [0,1).
\end{equation}
Then the sequence~$\mathcal{D}_n:=\mathcal{D}_n^{(1)}\setminus \mathcal{D}^{(2)}_n$ satisfies~\eqref{eq:van_Hove}. Moreover, if $\mathcal{D}^{(1)}_n$ satisfies~\eqref{eq:additional_assumption_new01}, then so does $\mathcal{D}_n$.
\end{assertion}
\begin{proof}
Using the inclusion $(A\setminus B)\Delta (C\setminus D)\subseteq (A\Delta C)\cup(B\Delta D)$ we obtain, for every fixed $c\in\mathbb{Z}^d$,
\begin{multline*}
\frac{\#(\mathcal{D}_n\Delta (\mathcal{D}_n+c))}{\#\mathcal{D}_n}=\frac{\#((\mathcal{D}_n^{(1)}\setminus \mathcal{D}^{(2)}_n)\Delta ((\mathcal{D}_n^{(1)}+c)\setminus (\mathcal{D}^{(2)}_n+c)))}{\#\mathcal{D}_n}\\
\leq \frac{\#(\mathcal{D}_n^{(1)}\Delta (\mathcal{D}_n^{(1)}+c))}{\#\mathcal{D}^{(1)}_n}\frac{\#\mathcal{D}^{(1)}_n}{\#\mathcal{D}_n}+\frac{\#(\mathcal{D}_n^{(2)}\Delta (\mathcal{D}_n^{(2)}+c))}{\#\mathcal{D}^{(2)}_n}\frac{\#\mathcal{D}^{(2)}_n}{\#\mathcal{D}_n}.
\end{multline*}
In view of~\eqref{eq:difference_assump},
\begin{equation}\label{prop:example1-1_proof}
0\leq \limsup_{n\to\infty}\frac{\#\mathcal{D}^{(2)}_n}{\#\mathcal{D}_n}\leq \limsup_{n\to\infty}\frac{\#\mathcal{D}^{(1)}_n}{\#\mathcal{D}_n}=\limsup_{n\to\infty}\frac{\#\mathcal{D}^{(1)}_n}{\#\mathcal{D}^{(1)}_n-\#\mathcal{D}^{(2)}_n}<\infty,
\end{equation}
and we see that $\mathcal{D}_n$ satisfies~\eqref{eq:van_Hove}.

If $\mathcal{D}^{(1)}_n$ satisfies~\eqref{eq:additional_assumption_new01}, then, for every $a,b\in\N$ and $i,j=1,\ldots,d$, $i\neq j$, it holds that for all $n\in\N$,
$$
\frac{\#(\mathcal{D}_n\cap \mathbb{Z}_i(a)\cap \mathbb{Z}_j(b))}{\#\mathcal{D}_n}\leq \frac{\#(\mathcal{D}^{(1)}_n\cap \mathbb{Z}_i(a)\cap \mathbb{Z}_j(b))}{\#\mathcal{D}_n^{(1)}}\frac{\#\mathcal{D}_n^{(1)}}{\#\mathcal{D}_n}\leq \frac{K}{ab}\sup_{n\in\N}\frac{\#\mathcal{D}_n^{(1)}}{\#\mathcal{D}_n}=: \frac{K'}{ab},
$$
where we used~\eqref{prop:example1-1_proof} for the last passage.
\end{proof}

With minimal changes, the above proof leads to the following.
\begin{assertion}\label{prop:intersection}
Let $\mathcal{D}^{(1)}_n$ and $\mathcal{D}^{(2)}_n$ be two sequences of sets satisfying~\eqref{eq:van_Hove}. Suppose further that
\begin{equation*}
\limsup_{n\to\infty}\frac{\#(\mathcal{D}^{(1)}_n\cup\mathcal{D}^{(2)}_n)}{\#(\mathcal{D}^{(1)}_n\cap \mathcal{D}^{(2)}_n)}<\infty.
\end{equation*}
Then the sequence~$\mathcal{D}_n:=\mathcal{D}_n^{(1)}\cap  \mathcal{D}^{(2)}_n$ satisfies~\eqref{eq:van_Hove}. Moreover, if $\mathcal{D}^{(1)}_n$ or $\mathcal{D}^{(2)}_n$ satisfies~\eqref{eq:additional_assumption_new01}, then $\mathcal{D}_n$ satisfies~\eqref{eq:additional_assumption_new01} as well.
\end{assertion}

\appendix

\section{On the regular growth condition for discrete  domains}\label{sec:appendix}

The following definition can be found on p.~173 in \cite{Bulinski+Shashkin}.

\begin{define}\label{def:Bulinski+Shashkin}
A sequence of finite sets $\mathcal{D}_n\subset \mathbb{Z}^d$ is said to  be regularly growing to infinity if as $n\to\infty$,
\begin{equation}
\label{eq:def_Bulinski+Shashkin}
   \#\mathcal{D}_n\to \infty\quad\text{and}\quad \frac{\#(\mathcal{D}^{1}_n\setminus \mathcal{D}_n)}{\#\mathcal{D}_n}\to 0,
\end{equation}
where for $A\subset\mathbb{Z}^d$ and $p\in\mathbb{N}$, we denote by
$$
A^{p}:=\{x=(x_1,\ldots,x_d)\in\mathbb{Z}^d:{\rm dist}(x,A)\leq p\},
$$
and ${\rm dist}$ is the supremum metric on $\mathbb{Z}^d$.
\end{define}

\begin{assertion}\label{prop:def_equivalence}
Assume that $\mathcal{D}_n\subset\mathbb{Z}^d$ is a sequence of finite sets and $\#\mathcal{D}_n\to \infty$ as $n\to\infty$. The following statements are equivalent:
\begin{enumerate}[label={\rm(\roman{*})},ref=(\roman{*})]
\item\label{[(i)]} Condition~\eqref{eq:van_Hove} holds for all $c\in\mathbb{Z}^d$.
\item\label{[(ii)]} Condition~\eqref{eq:van_Hove} holds for $c=\pm e_k$, $k=1,\ldots,d$.
\item\label{[(iii)]} Condition~\eqref{eq:van_Hove} holds for $c=e_k$, $k=1,\ldots,d$.
\item\label{[(iv)]} The sequence $\mathcal{D}_n$ is regularly growing.
\end{enumerate}
\end{assertion}

\begin{proof}
Condition \ref{[(i)]} trivially implies condition \ref{[(ii)]}, and \ref{[(ii)]} clearly implies \ref{[(iii)]}. The fact that \ref{[(iii)]}$\Longrightarrow$\ref{[(ii)]} follows from
$$
\#((\mathcal{D}_n-e_k)\Delta \mathcal{D}_n)=\#(((\mathcal{D}_n-e_k)\Delta \mathcal{D}_n)+e_k)=\#(\mathcal{D}_n\Delta (\mathcal{D}_n+e_k))=\#((\mathcal{D}_n+e_k)\Delta \mathcal{D}_n).
$$
We now prove that \ref{[(ii)]}$\Longrightarrow$\ref{[(iv)]}. Note that
$$
\mathcal{D}^{1}_n=\bigcup_{k=1}^{d}(\mathcal{D}_n\pm e_k).
$$
Thus,
$$
\frac{\#(\mathcal{D}^{1}_n\setminus \mathcal{D}_n)}{\#\mathcal{D}_n}\leq \sum_{k=1}^{d}\frac{\#((\mathcal{D}_n\pm e_k)\setminus \mathcal{D}_n)}{\#\mathcal{D}_n}\leq \sum_{k=1}^{d}\frac{\#((\mathcal{D}_n\pm e_k)\Delta \mathcal{D}_n)}{\#\mathcal{D}_n}.
$$
The right-hand side converges to $0$, since by~\eqref{eq:van_Hove} every summand converges to $0$.

We proceed to the proof of \ref{[(iv)]}$\Longrightarrow$\ref{[(i)]}. Assume that~\eqref{eq:def_Bulinski+Shashkin} holds and fix $ c\in\mathbb{Z}^d$. Using the inclusion $A\setminus B\subset (A\setminus C)\cup(C\setminus B)$ which holds for any sets $A,B,C$, we conclude that
$$
(\mathcal{D}_n+c)\Delta \mathcal{D}_n\subset \bigcup_j \left((\mathcal{D}_n+u_j)\setminus (\mathcal{D}_n+v_j)\right),
$$
where the union is finite and for every index $j$, $u_j-v_j=\pm e_{k_j}$ for some $k_j\in\{1,\ldots,d\}$. Since $\#\mathcal{D}_n=\#(\mathcal{D}_n+x)$ for every $x\in\mathbb{Z}^d$, it suffices to check that, for every $j$,
$$
\lim_{n\to\infty}\frac{\#\left((\mathcal{D}_n+u_j)\setminus (\mathcal{D}_n+v_j)\right)}{\#(\mathcal{D}_n+v_j)}=0,
$$
but this follows from the inclusion $(\mathcal{D}_n+u_j)=(\mathcal{D}_n+v_j\pm e_{k_j})\subset (\mathcal{D}_n+v_j)^1$ and the fact that if~\eqref{eq:def_Bulinski+Shashkin} holds for a sequence $\mathcal{D}_n$, it also holds for the shifted sequence $\mathcal{D}_n+x$, for every fixed $x\in\mathbb{Z}^d$.
\end{proof}

The following result is a combination of Proposition~\ref{prop:def_equivalence} and Lemma 1.5 in \cite{Bulinski+Shashkin}. In some cases, it is useful for checking~\eqref{eq:def_Bulinski+Shashkin}.
\begin{assertion}\label{prop:continuous-discrete}
Assume that $V_n$, $n\in \N$, is a sequence of bounded measurable subsets of $\mathbb{R}^d$ satisfying the so-called van Hove condition, meaning that for every $\varepsilon>0$
\begin{equation}
\label{eq:van_hove_continuous}
   \lim_{n\to\infty}\frac{{\rm Vol}(\partial V_n\oplus B^d_{\varepsilon}(0))}{{\rm Vol}(V_n)}=0,
\end{equation}
where $\partial V_n$ is the topological boundary of $V_n$. Then the sequence $\mathcal{D}_n:=V_n\cap \mathbb{Z}^d$ satisfies~\eqref{eq:def_Bulinski+Shashkin}.
% and $\#\mathcal{D}_n\to\infty$ as $n\to\infty$.
\end{assertion}

Our last auxiliary result provides sufficient conditions for~\eqref{eq:additional_assumption_new01}. It has been used in the proof of Proposition~\ref{prop:example2}.

\begin{assertion}\label{prop:add_assump_suff2}
Assume that there exist two sequences $(s_{1}(n),\ldots,s_{d}(n))_{n\in\N}$ and $(c_{1}(n),\ldots,c_{d}(n))_{n\in\N}$
of nonnegative integers such that the rectangle
$$
\Pi_n:=\left(\varprod_{i=1}^{d}[c_i(n),c_i(n)+s_i(n)]\right)\bigcap \N^d
$$
satisfies
\begin{equation}\label{eq:square_assumption}
\#\mathcal{D}_n\subset \Pi_n\quad\text{and}\quad \overline{C}:=\sup_{n\in\N}\frac{\#\Pi_n}{\#\mathcal{D}_n}<\infty.
\end{equation}
Then~\eqref{eq:additional_assumption_new01} holds. More generally, if~\eqref{eq:additional_assumption_new01} holds with $\mathcal{D}_n$ replaced by {\it some} set $\Pi_n$ which satisfies~\eqref{eq:square_assumption}, then~\eqref{eq:additional_assumption_new01} holds for $\mathcal{D}_n$.
\end{assertion}
\begin{proof}
Fix $i,j=1,\ldots,d$, $i\neq j$. If~\eqref{eq:square_assumption} holds, then for all $n\in\N$ and all $a,b\in\N$ it holds
\begin{equation*}
\frac{\#(\mathcal{D}_n\cap \mathbb{Z}_{i}(a)\cap \mathbb{Z}_{j}(b))}{\#\mathcal{D}_n}\leq \frac{\#(\Pi_n\cap \mathbb{Z}_{i}(a)\cap\mathbb{Z}_j(b))}{\#\mathcal{D}_n}\leq \overline{C}\frac{\#(\Pi_n\cap \mathbb{Z}_{i}(a)\cap\mathbb{Z}_j(b))}{\#\Pi_n}.
\end{equation*}
Since $i\neq j$, we obtain
$$
\frac{\#(\Pi_n\cap \mathbb{Z}_{i}(a)\cap\mathbb{Z}_j(b))}{\#\Pi_n}\leq\frac{1}{s_i(n)+1}\left\lfloor \frac{s_i(n)+1}{a}\right\rfloor \frac{1}{s_j(n)+1}\left\lfloor \frac{s_j(n)+1}{b}\right\rfloor\leq \frac{1}{ab}
$$
and the desired estimate holds true with $K=\overline{C}$.
\end{proof}

\section*{Acknowledgments}
This project has received funding from the European Research Council (ERC) under the European Union's Horizon 2020 research and innovation programme under the Grant Agreement No.\ 759702 and from Centre Henri Lebesgue, programme ANR-11-LABX-0020-0. ZK was supported by the German Research Foundation under Germany's Excellence Strategy  EXC 2044 -- 390685587, Mathematics M\"unster: Dynamics - Geometry - Structure.
AM was supported by UC Berkeley Economics/Haas in the framework of the U4U program. AM gratefully acknowledges the financial support and hospitality of the University of Angers during his stay in December 2022--March 2023.

\end{document}